\documentclass[a4paper,11pt]{amsart}

\usepackage{amsmath,amssymb,amsthm,mathtools,tikz,comment,bbm,cprotect}
\usepackage[margin=2.5cm]{geometry}
\usepackage{thm-restate} 
\usepackage{enumerate}
\usepackage[shortlabels]{enumitem}
\usepackage{hyperref}
\usepackage[noadjust]{cite}

\theoremstyle{plain}
\newtheorem{thm}{Theorem}[section]

\newtheorem{lem}[thm]{Lemma}

\theoremstyle{definition}
\newtheorem{defn}[thm]{Definition}

\newtheorem*{remark*}{Remark}
\newtheorem*{thm*}{Theorem}
\newtheorem*{defn*}{Definition}
\newtheorem*{lem*}{Lemma}
\newtheorem*{cor*}{Corollary}

\newcommand{\N}{\mathbb{N}}
\newcommand{\Z}{\mathbb{Z}}

\newcommand{\R}{\mathbb{R}}
\newcommand{\C}{\mathbb{C}}
\newcommand{\F}{\mathbb{F}}
\newcommand{\one}{\mathbbm{1}}

\DeclareMathOperator{\supp}{supp}

\newcommand{\Pset}{\mathcal{P}}
\newcommand{\Pfin}{\Pset_{\textup{fin}}}

\setenumerate{label=\textup{(\roman*)}} 

\title{Invariant means, coset correct means and the Hahn--Banach theorem}
\author{Armando Martino}
\address[A.\ Martino]{Mathematical Sciences, University of Southampton, University Road, Southampton SO17 1BJ, United Kingdom}
\email{A.Martino@soton.ac.uk}
\author{Motiejus Valiunas}
\address[M.\ Valiunas]{Instytut Matematyczny, Uniwersytet Wroc{\l}awski, plac Grunwaldzki 2, 50-384 Wroc{\l}aw, Poland}
\email{motiejus.valiunas@math.uni.wroc.pl}

\begin{document}

\begin{abstract}
    We show that, in ZF set theory, the existence of ``coset correct means'' on all groups is equivalent to the Hahn--Banach Theorem, as well as to the existence of invariant means on abelian groups.
\end{abstract}
\maketitle

\section{Introduction}

There are various ways to measure subsets of groups: Haar measures on compact groups, random walk means on finitely generated groups etc.  Most of such means are finitely additive and admit some sort of invariance under translation.  Notably, one may consider \emph{amenable} groups~$G$, which are those admitting an invariant mean $\mu \colon \Pset(G) \to [0,1]$ as defined below.  Motivated by this, in \cite{MartinoValiunas} we introduce a weakening of such a mean $\mu$, by requiring left invariance only on cosets of subgroups of $G$.

\begin{defn}
    Let $G$ be a group.  A function $\mu \colon \Pset(G) \to [0,1]$ is called an \emph{invariant mean} on $G$ if it satisfies the following properties:
    \begin{enumerate}[(i)]
        \item \label{it:inv-1} $\mu(G) = 1$;
        \item \label{it:inv-add} $\mu(A \sqcup B) = \mu(A) + \mu(B)$ for any disjoint $A,B \subseteq G$;
        \item \label{it:inv-am} $\mu(gA) = \mu(A)$ for any $g \in G$ and $A \subseteq G$.
    \end{enumerate}
    We may weaken \ref{it:inv-am} by requiring it only for subgroups $A \leq G$.  Specifically, a function $\mu \colon \Pset(G) \to [0,1]$ is called a \emph{coset correct mean} (\emph{CCM}) on $G$ if it satisfies \ref{it:inv-1}, \ref{it:inv-add} and
    \begin{enumerate}
        \item[(iii')] \label{it:inv-CCM} $\mu(gH) = \mu(H)$ for any $g \in G$ and any subgroup $H \leq G$.
    \end{enumerate}
\end{defn}

In \cite[Theorem~A.3]{MartinoValiunas}, we use the Hahn--Banach Theorem for real normed vector spaces to construct a CCM on any group.
The aim of this note is to prove the converse of such a statement:
\begin{thm} \label{thm:main}
    The following statements are equivalent in the axiomatic framework of Zermelo--Fraenkel set theory:
    \begin{enumerate}[label=\textup{(\roman*)}]
        \item \label{it:CCM} Every group admits a coset correct mean.
        \item \label{it:CCM-ZX} Every free abelian group admits a coset correct mean.
        \item \label{it:am} Every abelian group admits an invariant mean.
        \item \label{it:HB} \emph{The Hahn--Banach Theorem}: for any normed vector space $(V,\left\|\cdot\right\|)$ over $\F \in \{ \R,\C \}$, any subspace $S \leq V$ and any bounded linear functional $f_S\colon S \to \F$, there exists a bounded linear functional $f\colon V \to \F$ such that $f|_S = f_S$ and $\| f \|_{V^*} = \| f_S \|_{S^*}$.
    \end{enumerate}
\end{thm}

The implications \ref{it:CCM} $\Rightarrow$ \ref{it:CCM-ZX} and \ref{it:am} $\Rightarrow$ \ref{it:CCM-ZX} are trivial, and we proved \ref{it:HB} $\Rightarrow$ \ref{it:CCM} in \cite[Theorem~A.3]{MartinoValiunas}.  We prove \ref{it:CCM-ZX} $\Rightarrow$ \ref{it:HB} in Section~\ref{sec:CCM=>HB}, and \ref{it:HB} $\Rightarrow$ \ref{it:am} in Section~\ref{sec:HB=>am}.

\subsection*{Acknowledgements}
The second author was partially supported by the National Science Centre (Poland) grants No.\ 2022/47/D/ST1/00779 and 2025/58/E/ST1/00143.

\subsection*{AI declaration}
No form of AI was used in the preparation of this article.

\section{CCMs imply Hahn--Banach}
\label{sec:CCM=>HB}

We start by showing the following result.

\begin{thm} \label{thm:CCM=>HB}
    Suppose that any free abelian group admits a CCM.  Then the Hahn--Banach Theorem holds.
\end{thm}

We will use the following restatement of the Hahn--Banach Theorem.  Here, by a \emph{probability charge} on a Boolean algebra $(X,\vee,\wedge,\neg,0,1)$ we mean a function $\nu\colon X \to [0,1]$ such that $\nu(1) = 1$ and $\nu(x \vee y) = \nu(x)+\nu(y)$ whenever $x \wedge y = 0$.  Specifically, if $X = \Pset(S)$ for a set $S$, then a probability charge on $X$ is a function $\nu\colon X \to [0,1]$ such that $\nu(S) = 1$ and $\nu(A \sqcup B) = \nu(A) + \nu(B)$ for any disjoint $A,B \subseteq S$.  We assume that all our Boolean algebras $X$ are \emph{non-trivial}, i.e.\ $0 \neq 1$ in $X$.

\begin{thm}[see the equivalence of (HB7) and (HB13) in \cite{Schechter}]
    The Hahn--Banach Theorem is equivalent (in ZF set theory) to the following statement: every Boolean algebra admits a probability charge. \qed
\end{thm}

The idea of the proof of Theorem~\ref{thm:CCM=>HB} is as follows.  By using the fact that every finite Boolean algebra has a canonical mean obtained by assigning equal weight to all its atoms, we can try to define a (finitely additive, non-constant) mean on an arbitrary Boolean algebra $X$ by taking an ``average'' over all the means over its finite subalgebras.  The following result gives a way to measure subsets of the set of finite subsets of $X$, which in turn gives a way to take averages over finite subsets of $X$ by integrating.

Given a set $X$, we write $\Pfin(X)$ for the set of all finite subsets of $X$.

\begin{lem} \label{lem:CCM=>nu}
    Suppose that any free abelian group admits a CCM, and let $X$ be a set.  Then $\Pset(\Pfin(X))$ admits a probability charge $\nu = \nu_X$ satisfying $\nu(\{ F \in \Pfin(X) \mid F_0 \subseteq F \}) = 1$ for all $F_0 \in \Pfin(X)$.
\end{lem}

\begin{proof}
    Note that $G = \{ f \in \Z^X \mid \left| \supp(f) \right| < \infty \} \cong \bigoplus_{x \in X} \Z$ is free abelian, and so admits a CCM $\mu$.  Given $S \subseteq \Pfin(X)$, define
    \[
    \nu(S) \coloneqq \mu(\{ f \in G \mid \supp(f) \in S \}).
    \]
    Then $\nu\coloneqq \Pset(\Pfin(X)) \to [0,1]$ satisfies $\nu(\Pfin(X)) = 1$.  Moreover, $\nu$ is finitely additive: if $S,T \subseteq \Pfin(X)$ are disjoint, then so are $\{ f \in G \mid \supp(f) \in S \}$ and $\{ f \in G \mid \supp(f) \in T \}$, and hence $\nu(S \cup T) = \nu(S) + \nu(T)$ by finite additivity of $\mu$.

    Now note that for a given $F_0 \in \Pfin(X)$, the set $\{ F \in \Pfin(X) \mid F_0 \subseteq F \}$ is the complement of $\bigcup_{x \in F_0} \{ F \in \Pfin(X) \mid x \notin F \}$, and so it is enough to show that $\nu(\{ F \in \Pfin(X) \mid x \notin F \}) = 0$ for all $x \in X$. But indeed, we have
    \[
    \nu(\{ F \in \Pfin(X) \mid x \notin F \}) = \mu(\{ f \in G \mid x \notin \supp(f) \}) = \mu(\ker(\varphi_x)) = 0,
    \]
    where $\varphi_x\colon G \to \Z$ is the surjective homomorphism defined by $\varphi_x(f) = f(x)$, and so the last equality follows since $\ker(\varphi_x)$ is an infinite-index subgroup of $G$ and $\mu$ is a CCM.
\end{proof}

The following properties of finitely generated boolean algebras are well-known.  Recall that an \emph{atom} in a Boolean algebra $X$ is a non-zero element $x \in X$ such that $y \neq y \wedge x$ for any $y \in X \setminus \{0,x\}$.

\begin{lem} \label{lem:finite-Boolean}
    Given a Boolean algebra $X$ and a finite subset $F \subseteq X$, the subalgebra $\langle F \rangle$ generated by $F$ is finite and admits a unique probability charge $m_{\langle F \rangle}$ that gives the same weight to each atom.
\end{lem}

\begin{proof}
    $\langle F \rangle$ is finite by \cite[Chapter~11, Corollary~2]{HalmosGivant}.

    For the existence of $m_{\langle F \rangle}$, note every finite Boolean algebra is isomorphic to $\Pset(S)$ for some finite set $S$ \cite[Chapter~15, Corollary~1]{HalmosGivant}, and the function $m_S'\colon \Pset(S) \to [0,1]$ defined by $m_S'(T) \coloneqq \frac{|T|}{|S|}$ is clearly a probability charge.  For uniqueness, note that since the atoms in $\Pset(S)$ are singletons, $m_S'$ gives the same weight (equal to $\frac{1}{|S|}$) to each of them, and by finite additivity it is the unique probability charge satisfying this property.
\end{proof}

Now we define a way of integrating over $Y = \Pfin(X)$ for a Boolean algebra $X$, as follows.

\begin{defn} \label{defn:int}
    Given a set $Y$, a probability charge $\nu$ on $\Pset(Y)$ and a bounded function $f\colon Y \to [0,\infty)$, let
    \[
    \int_Y f(y) \,\mathrm{d}\nu(y) \coloneqq \inf \sum_{i=1}^n a_i \nu(B_i),
    \]
    where the infimum ranges over all $a_1,\ldots,a_n \in [0,\infty)$ and all partitions $Y = B_1 \sqcup \cdots \sqcup B_n$ such that $f(y) \leq a_i$ for all $y \in B_i$.
\end{defn}

The following properties are well-known (see \cite[Chapter~III]{DunfordSchwartz}) and can be proved within ZF set theory.  As a compromise between brevity and self-containment, we give a brief sketch to convince the reader that no form of the Axiom of Choice is necessary here.

\begin{lem} \label{lem:int}
    Let $Y$ be a set and let $\nu$ be a probability charge on $\Pset(Y)$.
    \begin{enumerate}[label=\textup{(\roman*)}]
        \item \label{it:int-sup} For any bounded function $f\colon Y \to [0,\infty)$, the integral $\int_Y f(y) \,\mathrm{d}\nu(y)$ is finite and equal to $\sup \sum_{i=1}^m a_i \nu(C_i)$, where the supremum ranges over all $a_1,\ldots,a_m \in [0,\infty)$ and all partitions $Y = C_1 \sqcup \cdots \sqcup C_m$ such that $f(y) \geq a_i$ for all $y \in C_i$.
        \item \label{it:int-add} For any bounded functions $f,g \colon Y \to [0,\infty)$, we have
        \[
        \int_Y (f+g)(y) \,\mathrm{d}\nu(y) = \int_Y f(y) \,\mathrm{d}\nu(y) + \int_Y g(y) \,\mathrm{d}\nu(y).
        \]
        \item \label{it:int-nu1} If $T \subseteq Y$ satisfies $\nu(T) = 1$, then $\int_Y f(y) \,\mathrm{d}\nu(y) = \int_T f(y) \,\mathrm{d}\nu(y)$ for any bounded function $f\colon Y \to [0,\infty)$.
    \end{enumerate}
\end{lem}

\begin{proof}[Proof (sketch)]
    One may rephrase Definition~\ref{defn:int} to say that $\int_Y f(y) \,\mathrm{d}\nu(y)$ is the infimum of $I(f_1) \coloneqq \sum_{t \in f_1(Y)} t \cdot \nu(f_1^{-1}(t))$ over all simple functions $f_1 \geq f$, where $g\colon Y \to [0,\infty)$ is called a \emph{simple function} if $g(Y)$ is finite.  Similarly, part~\ref{it:int-sup} is equivalent to saying that $\int_Y f(y) \,\mathrm{d}\nu(y)$ is finite and equal to the supremum of $I(f_0)$ over all simple functions $f_0 \leq f$.  
    Part~\ref{it:int-sup} then follows if we can find simple functions $f_0,f_1 \colon Y \to [0,\infty)$ such that $f_0 \leq f \leq f_1$ with $I(f_1)-I(f_0)$ arbitrarily small.  However, this can be done by taking $f_0(y) = \frac{1}{n} \cdot \lfloor n \cdot f(y) \rfloor$ and $f_1(y) = \frac{1}{n} \cdot \lfloor n \cdot f(y) + 1 \rfloor$ for $n \geq 1$ large enough: in that case, taking $B_i = f^{-1}([\frac{i}{n},\frac{i+1}{n}))$ gives a partition $Y = B_0 \sqcup \cdots \sqcup B_m$, where $m = \lfloor n \cdot \sup(f) \rfloor$, such that $f_j(y) = \frac{i+j}{n}$ for $y \in B_i$ and therefore
    \[
    I(f_j) = \sum_{i=1}^m \frac{i+j}{n} \nu(B_i) = \frac{j}{n} + \sum_{i=1}^m \frac{i}{n} \nu(B_i)
    \]
    for $j \in \{0,1\}$, implying that $I(f_1)-I(f_0) = \frac{1}{n}$.

    Part~\ref{it:int-add} follows since if $f_0,f_1,g_0,g_1\colon Y \to [0,\infty)$ are simple functions such that $f_0 \leq f \leq f_1$ and $g_0 \leq g \leq g_1$, then we have $f_0+g_0 \leq f+g \leq f_1+g_1$: finite additivity of $\nu$ ensures that the part~\ref{it:int-add} holds for simple functions, and so the result follows since $\sup(f_1-f_0)$ and $\sup(g_1-g_0)$ can be taken to be arbitrarily small.

    Part~\ref{it:int-nu1} follows since we have
    \[
    \int_Y f(y) \,\mathrm{d}\nu(y) = \int_T f(y) \,\mathrm{d}\nu(y) + \int_{Y \setminus T} f(y) \,\mathrm{d}\nu(y)
    \]
    by finite additivity of $\nu$, and the second term is bounded above by $\sup(f) \cdot \nu(Y \setminus T) = 0$.
\end{proof}

\begin{proof}[Proof of Theorem~\ref{thm:CCM=>HB}]
    Let $X$ be any Boolean algebra, and let $\nu = \nu_X$ be the probability charge given by Lemma~\ref{lem:CCM=>nu}.  Given any $F \in \Pfin(X)$, let $m_{\langle F \rangle}$ be as in Lemma~\ref{lem:finite-Boolean}, and define $\overline{m}_{\langle F \rangle}\colon X \to [0,1]$ by setting $\overline{m}_{\langle F \rangle}(x) = m_{\langle F \rangle}(x)$ if $x \in \langle F \rangle$ and $\overline{m}_{\langle F \rangle}(x) = 0$ otherwise.  Then define $m\colon X \to [0,1]$ by setting
    \[
    m(x) = \int_{\Pfin(X)} \overline{m}_{\langle F \rangle}(x) \,\mathrm{d}\nu(F).
    \]
    We claim that $m$ is a probability charge.

    Note that we have $1 \in \langle F \rangle$ and therefore $\overline{m}_{\langle F \rangle}(1) = m_{\langle F \rangle}(1) = 1$ for all $F \in \Pfin(X)$, implying that $m(1) = 1$.  For finite additivity, let $x,y \in X$ be such that $x \wedge y = 0$.  By Lemma~\ref{lem:CCM=>nu}, we have $\nu(T) = 1$, where $T = \{ F \in \Pfin(X) \mid x,y \in F \}$.  Since $\overline{m}_{\langle F \rangle}(x \vee y) = m_{\langle F \rangle}(x \vee y) = m_{\langle F \rangle}(x) + m_{\langle F \rangle}(y) = \overline{m}_{\langle F \rangle}(x) + \overline{m}_{\langle F \rangle}(y)$ whenever $F \in T$ by the finite additivity of $m_{\langle F \rangle}$, it follows from Lemma~\ref{lem:int} that
    \[
    m(x \vee y) = \int_{T} \overline{m}_{\langle F \rangle}(x \vee y) \,\mathrm{d}\nu(F) = \int_{T} \overline{m}_{\langle F \rangle}(x) \,\mathrm{d}\nu(F) + \int_{T} \overline{m}_{\langle F \rangle}(y) \,\mathrm{d}\nu(F) = m(x) + m(y),
    \]
    as required.
\end{proof}

\section{Hahn--Banach implies invariant means}
\label{sec:HB=>am}

In this section we prove the following result.

\begin{lem} \label{lem:HB=>am}
    Suppose that the Hahn--Banach Theorem holds.  Then any abelian group admits an invariant mean.
\end{lem}

\begin{proof}
    Let $G$ be an abelian group, and let $V = \ell^\infty(G)$.  Let $Y \leq V$ be the subspace spanned by the elements $p(g,v) \coloneqq s_gv-v$ for $g \in G$ and $v \in V$, where $s_g\colon V \to V$ is the shift map defined by $(s_gv)_h = v_{gh}$ for all $h \in G$.  Let $\one \in V$ be the element defined by $\one_g = 1$ for all $g \in G$.

    We claim that $\| \one - v \| \geq 1$ for all $v \in Y$.  Indeed, suppose for contradiction that $\| \one - v \| = 1-\varepsilon$ for some $v \in Y$ and some $\varepsilon > 0$, and therefore $v_g \geq \varepsilon$ for all $g \in G$.  Since $v \in Y$, we have $v = \sum_{i=1}^n p(g_i,v^{(i)})$ for some $g_1,\ldots,g_n \in G$ and $v^{(1)},\ldots,v^{(n)} \in V$, and therefore
    \[
        v_h = \sum_{i=1}^n p(g_i,v^{(i)})_h = \sum_{i=1}^n v^{(i)}_{g_ih} - v^{(i)}_h \geq \varepsilon \qquad\text{for all } h \in G.
    \]
    Let $D = \max \{ \|v^{(i)}\| \mid 1 \leq i \leq n \}$, and let $N \in \N$ be arbitrary.  We denote by $X_N$ the set $\{ (x_1,\ldots,x_n) \in \Z^n \mid 0 \leq x_i \leq N-1 \}$, and given $\mathbf{x} = (x_1,\ldots,x_n) \in \Z^n$ we write $\mathbf{g^x}$ for the element $g_1^{x_1} \cdots g_n^{x_n}$.  We then have:
    \begin{align*}
        N^n \cdot \varepsilon &\leq \sum_{\mathbf{x} \in X_N} v_{\mathbf{g^x}} = \sum_{\mathbf{x} \in X_N} \sum_{i=1}^n v^{(i)}_{g_i\mathbf{g^x}} - v^{(i)}_{\mathbf{g^x}} = \sum_{i=1}^n \sum_{\substack{0 \leq x_j \leq N-1 \\ \text{for }1 \leq j \leq n}} v^{(i)}_{g_i\mathbf{g^x}} - v^{(i)}_{\mathbf{g^x}} \\
        &= \sum_{i=1}^n \sum_{\substack{0 \leq x_j \leq N-1 \\ \text{for }1 \leq j \leq n,\, j \neq i}} v^{(i)}_{g_1^{x_1} \cdots g_{i-1}^{x_{i-1}} g_i^N g_{i+1}^{x_{i+1}} \cdots g_n^{x_n}} - v^{(i)}_{g_1^{x_1} \cdots g_{i-1}^{x_{i-1}} g_{i+1}^{x_{i+1}} \cdots g_n^{x_n}} \\
        &\leq \sum_{i=1}^n \sum_{\substack{0 \leq x_j \leq N-1 \\ \text{for }1 \leq j \leq n,\, j \neq i}} 2\|v^{(i)}\| \leq n \cdot N^{n-1} \cdot 2D.
    \end{align*}
    But $N \in N$ was arbitrary, so taking any $N \geq \frac{2Dn}{\varepsilon}$ we get a contradiction.  We thus must have $\| \one - v \| \geq 1$ for all $v \in Y$, as claimed.

    Let $S \leq V$ be the subspace spanned by $Y$ and $\one$.  Since $\| \one - v \| \geq 1$ for all $v \in Y$, linearity implies that $\| t\one - v \| \geq |t|$ for all $v \in Y$ and $r \in \R$, and in particular any element of $S$ has a unique expression of the form $t\one+v$ for $t \in \R$ and $v \in Y$.  This shows that the linear functional $f_S\colon S \to \R$ defined by $f_S(t\one+v) = t$ is well-defined and bounded of norm $\leq 1$, and since $|f_S(\one)| = 1 = \|\one\|$, it has norm exactly $1$.

    Now by the Hahn--Banach Theorem, $f_S$ extends to a norm-$1$ linear functional $f\colon V \to \R$.  We define $\mu\colon \Pset(G) \to \R$ by setting $\mu(A) = f(1_A)$, where $1_A \in V$ is defined by $(1_A)_g = 1$ if $g \in A$ and $(1_A)_g = 0$ otherwise.  We claim that $\mu$ is an invariant mean on $G$.  Indeed:
    \begin{enumerate}[(i)]
        \item we have $\mu(G) = f(\one) = f_S(\one) = 1$;
        \item for disjoint $A,B \subseteq G$, we have $1_{A \sqcup B} = 1_A + 1_B$, so $\mu(A \sqcup B) = \mu(A) + \mu(B)$ by linearity of $f$;
        \item given $A \subset G$ and $g \in G$ we have $1_{gA}-1_A = p(g^{-1},1_A) \in Y$ and therefore $f(1_{gA}-1_A) = 0$, implying that $\mu(gA) = \mu(A)$.
    \end{enumerate}
    In order to show that $\mu$ is an invariant mean it is then enough to show that it takes values in $[0,1]$.  But indeed, given $A \subseteq G$ we have $1 - \mu(A) = \mu(G)-\mu(A) = \mu(A^c) = f(1_{A^c})$ and so $|1-\mu(A)| \leq \|1_{A^c}\| = 1$ since $\|f\| = 1$, implying that $\mu(A) \geq 0$; finite additivity then gives $\mu(A) = \mu(G)-\mu(A^c) \leq \mu(G) = 1$.
\end{proof}

\bibliographystyle{amsalpha}
\bibliography{ref}

\end{document}